\pdfoutput=1
\documentclass[letterpaper,12pt]{amsart}
\usepackage{amsmath,amssymb,amsthm}
\usepackage{amsaddr}
\usepackage{xurl}
\usepackage{hyperref}
\usepackage[capitalise]{cleveref}
\hypersetup{colorlinks=true,citecolor=red}
\usepackage{enumitem}
\setlistdepth{10}
\renewlist{itemize}{itemize}{10}
\setlist[itemize]{label=\textbullet}

\newtheorem{theorem}{Theorem}[section]
\newtheorem{lemma}[theorem]{Lemma}
\newtheorem{corollary}[theorem]{Corollary}
\newtheorem{proposition}[theorem]{Proposition}

\theoremstyle{definition}
\newtheorem{remark}[theorem]{Remark}
\newtheorem{definition}[theorem]{Definition}
\newtheorem{example}[theorem]{Example}

\newcommand\gen[1]{\langle{#1}\rangle}
\newcommand\N{\mathbb{N}}
\newcommand\Z{\mathbb{Z}}

\DeclareMathOperator{\size}{\mathsf{size}}
\DeclareMathOperator{\lc}{lc}
\DeclareMathOperator{\Fin}{\mathsf{Fin}}
\DeclareMathOperator{\Kdim}{\mathsf{Kdim}}
\DeclareMathOperator{\List}{\mathsf{List}}

\begin{document}

\title[Quantitative Hilbert's basis theorem and constructive Krull dimension]{A quantitative Hilbert's basis theorem and the constructive Krull dimension}
\author{Ryota Kuroki}
\email{kuroki-ryota128@g.ecc.u-tokyo.ac.jp}
\address{Graduate School of Mathematical Sciences, The University of Tokyo, 3-8-1 Komaba, Meguro-ku, Tokyo, 153-8914, Japan}

\begin{abstract}
In classical mathematics, Gulliksen has introduced the length of Noetherian modules, and Brookfield has determined the length of Noetherian polynomial rings.
Brookfield's result can be regarded as a quantitative version of Hilbert's basis theorem.
In this paper, based on the inductive definition of Noetherian modules in constructive algebra, we introduce a constructive version of the length called $\alpha$-Noetherian modules, and present a constructive proof of some results by Brookfield.
As a consequence, we obtain a new constructive proof of $\dim K[X_0,\ldots,X_{n-1}]<1+n$ and $\dim\mathbb{Z}[X_0,\ldots,X_{n-1}]<2+n$, where $K$ is a discrete field.

\end{abstract}


\maketitle
\section{Introduction}
In this paper, all rings are assumed to have an identity, and the term ``module'' refers to a left module.
In constructive arguments, an ordinal means a Cantor normal form (i.e., an ordinal less than $\varepsilon_0$). See \cite{CC06,Gri13,NFXG20,KNFX21,KNFX23} for type-theoretic treatments of Cantor normal forms.
In fact, we only need ordinals less than $\omega^\omega$ to obtain results on the Krull dimension, and such ordinals can be expressed as
\[\omega^{n-1}\cdot a_{n-1}+\cdots+\omega^1\cdot a_1+\omega^0\cdot a_0\quad(n\in\N,\ a_0,\ldots,a_{n-1}\in\N,\ a_{n-1}\ge1).\]

Hilbert's basis theorem is an important topic in constructive algebra.
In exploring constructive versions of Hilbert's basis theorem, several definitions of Noetherian rings have been considered \cite{BSB23}, including Richman--Seidenberg Noetherian rings \cite{Ric74,Sei74,MRR88}.
Among them, Jacobsson and Löfwall's one \cite[Definition 3.4]{JL91} and Coquand and Persson's one \cite[Section 3.1]{CP99} are (generalized) inductive definitions.

In this paper, we quantify the inductive definition and define $\alpha$-Noetherian rings for an ordinal $\alpha$. Then we constructively prove a quantitative version of Hilbert's basis theorem and present an application to the Krull dimension of polynomial rings.

In Section 2, based on the inductive definitions of Noetherianity \cite{JL91,CP99}, we define the notion of $\alpha$-Noetherian modules and $\alpha$-Noetherian rings for an ordinal $\alpha$.
Every discrete field is $1$-Noetherian, and $\Z$ is $\omega$-Noetherian.
If a commutative ring $A$ is $\alpha$-Noetherian for some $\alpha<\omega^n$, then $\Kdim A<n$ holds (\cref{N-K}).

In Section 3, we prove a quantitative version of Hilbert's basis theorem (\cref{qhbt}): if a ring $A$ is $\alpha$-Noetherian, then $A[X]$ is $(\omega\otimes\alpha)$-Noetherian, where $\otimes$ denotes the Hessenberg natural product.
This gives a new constructive proof of $\Kdim K[X_1,\ldots,X_n]<1+n$ and $\Kdim \Z[X_1,\ldots,X_n]<2+n$, where $K$ is a discrete field.

In classical mathematics, our main theorems have already been proved.
The notion of $\alpha$-Noetherian rings is essentially introduced by Gulliksen \cite{Gul73} and developed by Brookfield \cite{Bro02} in the form of the length $l(M)$ of a Noetherian module $M$.
In fact, a module $M$ is $\alpha$-Noetherian if and only if $M$ is Noetherian and $l(M)\le\alpha$ (\cref{length}).
For every Noetherian ring $A$, Brookfield \cite[Theorem 3.1]{Bro03} has proved that $l(A[X])=\omega\otimes l(A)$.
As noted in \cite{Bro03}, when $A$ is Noetherian, Brookfield's theorem implies $\Kdim A[X]=\Kdim A+1$ since $l(A)<\omega^n$ is equivalent to $\Kdim A<n$ \cite[Theorem 2.3]{Gul73}.
\section{\texorpdfstring{$\alpha$}{α}-Noetherian rings}
We first introduce some notation.
Let $[\alpha,\beta]:=\{\gamma:\alpha\le\gamma\le\beta\}$, $[\alpha,\beta):=\{\gamma:\alpha\le\gamma<\beta\}$, and $(\alpha,\beta]:=\{\gamma:\alpha<\gamma\le\beta\}$.
Let $\List S$ denote the set of finite lists of elements of a set $S$.
We may sometimes write a list $[x_0,\ldots,x_{n-1}]\in\List S$ as $[x_0,\ldots,x_{n-1}]_S$.
The expression $[]$ denotes the empty list.
For $\sigma=[x_0,\ldots,x_{n-1}]_S$ and $x\in S$, let $\sigma.x:=[x_0,\ldots,x_{n-1},x]_S$.
For a module $M$ over a ring $A$, let $\gen{x_0,\ldots,x_{n-1}}$ denote the submodule of $M$ generated by $x_0,\ldots,x_{n-1}\in M$.
Let $\gen{[x_0,\ldots,x_{n-1}]_M}:=\gen{x_0,\ldots,x_{n-1}}$.

We define $\alpha$-good lists and $\alpha$-Noetherian rings for an ordinal $\alpha$.
\begin{definition}
Let $M$ be a module over a ring $A$ and $\alpha$ be an ordinal.
    \begin{enumerate}
        \item A list $[x_0,\ldots,x_{n-1}]_M$ is called \emph{$(-1)$-good} (or simply \emph{good}), if $n\ge1$ and $x_{n-1}\in\gen{x_0,\ldots,x_{n-2}}$.
        Note that this definition is different from the definition of a good list in \cite[Section 3.1]{CP99}.
        \item A list $\sigma\in\List M$ is called \emph{$\alpha$-good} if for every $x\in M$, there exists $\beta\in[-1,\alpha)$ such that $\sigma.x$ is $\beta$-good.
        \item A module $M$ is called \emph{$\alpha$-Noetherian} if $[]_M$ is $\alpha$-good.
        A ring $A$ is called \emph{(left) $\alpha$-Noetherian} if it is $\alpha$-Noetherian as a left $A$-module.
    \end{enumerate}
\end{definition}
\begin{remark}
    The above definition of $\alpha$-Noetherian modules is a quantitative version of the following generalized inductive definition of Noetherian modules:
    \begin{enumerate}
        \item A list $[x_0,\ldots,x_{n-1}]_M$ is called \emph{good} if $n\ge1$ and $x_{n-1}\in\gen{x_0,\ldots,x_{n-2}}$.
        \item We inductively generate the predicate ``\emph{good bars $\sigma$}'' by the following constructors:
        \begin{enumerate}
            \item If $\sigma$ is good, then good bars $\sigma$.
            \item If good bars $\sigma.x$ for every $x$, then good bars $\sigma$.
        \end{enumerate}
        \item A module $M$ is called \emph{Noetherian} if good bars $[]\in\List M$.
    \end{enumerate}
    Inductive definitions of Noetherian rings have been introduced by Jacobsson, Löfwall \cite[Definition 3.4]{JL91}, Coquand, and Persson \cite[Section 3.1]{CP99}.
    Although the above inductive definition is similar to these, it is different from the definition in \cite{JL91} because the above definition does not contain negation, and it is also different from the one in \cite{CP99} (at least on the surface) because the above definition uses a stronger notion of goodness.
    We do not know whether the above definition of Noetherian modules is equivalent to the one in \cite{CP99}.
\end{remark}
\begin{remark}
    We also have the following alternative definition of $\alpha$-Noetherian modules.
    \begin{enumerate}
        \item A finitely generated submodule $N$ of $M$ is \emph{$\alpha$-blocked} if for every $x\in M$,
        \begin{enumerate}
            \item $x\in N$, or
            \item there exists $\beta\in[0,\alpha)$ such that $N+\gen{x}$ is $\beta$-blocked.
        \end{enumerate}
        \item A module $M$ is called \emph{$\alpha$-Noetherian} if $0\subseteq M$ is $\alpha$-blocked.
    \end{enumerate}
    The above definition of $\alpha$-blocked module is a quantitative version of the following modified negation-free definition of the blocked modules \cite[Definition 3.1]{JL91}:
    \begin{itemize}
        \item A finitely generated submodule $N$ of $M$ is \emph{blocked} if for every $x\in M$,
        \begin{enumerate}
            \item $x\in N$ or
            \item there exists $\beta\in[0,\alpha)$ such that $N+\gen{x}$ is blocked.
        \end{enumerate}
    \end{itemize}
\end{remark}
\begin{example}\label{ex-noe}
Let $K$ be a discrete field.
    \begin{enumerate}
        \item A ring is $0$-Noetherian if and only if it is trivial.
        \item A ring is $1$-Noetherian if and only if it is a discrete field.
        \item Let $n\in\N$. The rings $\Z/\gen{2^n}$ and $K[X]/\gen{X^n}$ are $n$-Noetherian.
        \item The rings $\Z$ and $K[X]$ are $\omega$-Noetherian.
    \end{enumerate}
\end{example}
More generally, we can define $\alpha$-Euclidean rings and prove that they are $\alpha$-Noetherian.
\begin{definition}
    Let $M$ be a module over a ring $A$ and $\alpha$ be an ordinal.
    \begin{enumerate}
        \item An element $x\in M$ is called \emph{$(-1)$-Euclidean} if $x=0$.
        \item An element $x\in M$ is called \emph{$\alpha$-Euclidean} if for every $y\in M$, there exist $\beta\in[-1,\alpha)$ and $\beta$-Euclidean element $z\in M$ such that $z-y\in\gen{x}$.
        \item A module $M$ is called \emph{$\alpha$-Euclidean} if for every $x\in M$, there exists $\beta\in[-1,\alpha)$ such that $x$ is $\beta$-Euclidean. A ring $A$ is called \emph{(left) $\alpha$-Euclidean} if it is $\alpha$-Euclidean as a left $A$-module.
    \end{enumerate}
\end{definition}
\begin{remark}
    In classical mathematics, Motzkin \cite[Section 2]{Mot49} has introduced a transfinite version of the Euclidean ring, and it has been studied by several authors \cite{Fle71,Sam71,Hib75,Hib77,Nag78,Nag85,Cla15,CNT19}.
    Non-commutative Euclidean rings are studied in \cite{Ore33,Coh61,Bru73}.
    Euclidean modules are studied in \cite{Len74,Rah02,LC14}.
    We note that Lenstra \cite{Len74} has treated all three generalizations of Euclidean rings.
\end{remark}
\begin{example}
    Let $K$ be a discrete field.
    \begin{enumerate}
        \item A ring is $0$-Euclidean if and only if it is a trivial ring.
        \item A ring is $1$-Euclidean if and only if it is a discrete field.
        \item Let $n\in\N$. The rings $\Z/\gen{2^n}$ and $K[X]/\gen{X^n}$ are $n$-Euclidean.
        \item The rings $\Z$ and $K[X]$ are $\omega$-Euclidean.
    \end{enumerate}
\end{example}
We use the following lemma to prove that every $\alpha$-Euclidean module is $\alpha$-Noetherian.
\begin{lemma}\label{E-N-lemma}
    Let $M$ be a module over a ring $A$, $\alpha$ be an ordinal, and $\sigma\in\List M$.
    If $\gen{\sigma}$ contains an $\alpha$-Euclidean element $x$, then $\sigma$ is $\alpha$-good.
\end{lemma}
\begin{proof}
    We prove this by induction on $\alpha$.
    \begin{itemize}
        \item Let $y\in M$. Since $x\in\gen{\sigma}$ is $\alpha$-Euclidean, there exists $\beta\in[-1,\alpha)$ and a $\beta$-Euclidean element $z\in M$ such that $z-y\in\gen{\sigma}$.
        \begin{enumerate}
            \item If $\beta=-1$, then $z=0$. Hence $y\in\gen{\sigma}$, and $\sigma.y$ is good.
            \item If $\beta\in[0,\alpha)$, then $z\in\gen{\sigma.y}$. Hence $\sigma.y$ is $\beta$-good by the inductive hypothesis.
        \end{enumerate}
    \end{itemize}
    Hence $\sigma$ is $\alpha$-good.
\end{proof}
\begin{theorem}[{Classically proved in \cite[Theorem 3.17]{Cla15}}]\label{E-N}
    Let $\alpha$ be an ordinal and $M$ be a module over a ring $A$.
    If $M$ is an $\alpha$-Euclidean module over a ring $A$, then $M$ is $\alpha$-Noetherian.
\end{theorem}
\begin{proof}
    Let $x\in M$. Then, there exists $\beta\in[-1,\alpha)$ such that $x$ is $\beta$-Euclidean.
        \begin{enumerate}
            \item If $\beta=-1$, then $x=0$. Hence $[x]$ is good.
            \item If $\beta\in[0,\alpha)$, then $[x]$ is $\beta$-good by \cref{E-N-lemma}.
        \end{enumerate}
    Hence $[]_M$ is $\alpha$-good.
\end{proof}
We next prove that a sequence indexed by $[0,\alpha]$ in an $\alpha$-Noetherian module contains a reversed good subsequence.
\begin{lemma}\label{N-K-lemma}
    Let $M$ be a module over a ring $A$, $\alpha$ be an ordinal, and $f:[0,\alpha)\to M$ be a function. Let $\beta\in[-1,\alpha)$.
    If there exist $n\in\N$ and a strictly decreasing sequence $\alpha_0,\ldots,\alpha_{n-1}\in(\beta,\alpha]$ such that $[f(\alpha_0),\ldots,f(\alpha_{n-1})]$ is $\beta$-good, then there exist $m\in\N$ and a strictly decreasing sequence $\alpha_{n},\ldots,\alpha_{n+m-1}\in[0,\alpha_{n-1})$ such that $[f(\alpha_0),\ldots,f(\alpha_{n+m-1})]$ is good.
\end{lemma}
\begin{proof}
    We prove this by induction on $\beta$.
    \begin{enumerate}
        \item If $\beta=-1$, then $[f(\alpha_0),\ldots,f(\alpha_{n-1})]$ is good.
        \item Let $\beta\in[0,\alpha)$. Since $[f(\alpha_0),\ldots,f(\alpha_{n-1})]$ is $\beta$-good, there exists $\beta'\in[-1,\beta)$ such that $[f(\alpha_1),\ldots,f(\alpha_{n-1}),f(\beta)]$ is $\beta'$-good.
        By the inductive hypothesis, there exist $m\in\N$ and a strictly decreasing sequence $\alpha_{n+1},\ldots,\alpha_{n+m}\in[0,\beta)$ such that \[[f(\alpha_0),\ldots,f(\alpha_{n-1}),f(\beta),f(\alpha_{n+1}),\ldots,f(\alpha_{n+m})]\] is good.\qedhere
    \end{enumerate}
\end{proof}
\begin{theorem}\label{N-K-aux}
    Let $M$ be a module over a ring $A$, $\alpha$ be an ordinal, $f:[0,\alpha)\to M$ be a function, and $\beta\in[0,\alpha)$. If $M$ is $\beta$-Noetherian, then there exist $m\in\N$ and a strictly decreasing sequence $\alpha_{0},\ldots,\alpha_{m-1}\in[0,\alpha)$ such that $[f(\alpha_0),\ldots,f(\alpha_{m-1})]$ is good.
\end{theorem}
\begin{proof}
    Let $n:=0$ in \cref{N-K-lemma}.
\end{proof}
We recall the following elementary characterization of the Krull dimension by Lombardi \cite[Définition 5.1]{Lom02}.
See \cite[Note historique]{Lom23} for the background of the constructive definition of the Krull dimension.
\begin{definition}
    Let $n\in\N$. A ring $A$ is of \emph{Krull dimension less than $n$} if for every $x_0,\ldots,x_{n-1}\in A$, there exists $e_0,\ldots,e_{n-1}\ge0$ such that
    \[x_0^{e_0}\cdots x_{n-1}^{e_{n-1}}\in\gen{x_0^{e_0+1},x_0^{e_0}x_1^{e_1+1},\ldots,x_0^{e_0}\cdots x_{n-2}^{e_{n-2}}x_{n-1}^{e_{n-1}+1}}.\]
    Let $\Kdim A< n$ denote the statement that $A$ is of Krull dimension less than $n$.
\end{definition}
The following relation between the notion of $\alpha$-Noetherian rings and the Krull dimension easily follows from \cref{N-K-aux}:
\begin{theorem}[{Classically proved in \cite[Theorem 2.3]{Gul73}}]\label{N-K}
    Let $n\in\N$ and $A$ be a commutative ring.
    If $A$ is $\alpha$-Noetherian for some $\alpha<\omega^n$, then $\Kdim A<n$.
\end{theorem}
\begin{proof}
    Let $x_{0},\ldots,x_{n-1}\in A$.
    Define $f:\omega^n\to A$ by $f(e_{n-1},\ldots,e_1,e_0):=x_0^{e_0}\cdots x_{n-1}^{e_{n-1}}$.
    We put the anti-lexicographic order on $\omega^n$ and identify it with $[0,\omega^n)$.
    By \cref{N-K-aux}, there exist $m\in\N$ and a strictly decreasing sequence $\alpha_{0},\ldots,\alpha_{m-1}\in[0,\omega^n)$ such that $[f(\alpha_0),\ldots,f(\alpha_{m-1})]$ is good. Hence $\Kdim A<n$.
\end{proof}
\begin{example}
    Since $\Z$ is $\omega$-Noetherian, every sequence $a_0,a_1,\ldots,a_\omega\in\Z$ has a reversed good subsequence. In particular, $x^0,x^1,\ldots,y\in\Z$ has a reversed good subsequence for every $x,y\in\Z$. This implies $\Kdim\Z<2$.
\end{example}
\section{The quantitative Hilbert's basis theorem}
In this section, we will prove the following quantitative version of Hilbert's basis theorem:
    \begin{itemize}
        \item If a module $M$ over a ring $A$ is $\alpha$-Noetherian, then the $A[X]$-module $M[X]$ is $(\omega\otimes\alpha)$-Noetherian, where $\otimes$ denotes the Hessenberg natural product.
    \end{itemize}
We use some basic facts about transfinite chomp described in \cite{HS02}.
\begin{definition}
    A set $S$ is called \emph{finite} if there exist $n\in\N$ and a surjection from $\{0,\ldots,n-1\}$ to $S$. Let $\Fin S$ denote the set of all finite subsets of a set $S$.
\end{definition}
\begin{definition}
    Let $\alpha,\beta$ be ordinals.
    We define a binary relation $\prec$ on $\alpha\times\beta$ by
    \[(\alpha_1,\beta_1)\prec(\alpha_0,\beta_0)
    :\equiv(\alpha_1<\alpha_0)\lor(\beta_1<\beta_0).\]
    For $x\in\alpha\times\beta$, we define a detachable subset $\{\prec x\}\subseteq\alpha\times\beta$ by
    \[\{\prec x\}:=\{y\in\alpha\times\beta:y\prec x\}.\]
    For $S\in\Fin(\alpha\times\beta)$, let
    \[
    \{\prec S\}:=\bigcap_{x\in S}\{\prec x\}.
    \]
    We define a decidable binary relation $<$ on $\Fin(\alpha\times\beta)$ by
    \[
    S<T:\equiv\{\prec S\}\subsetneq\{\prec T\}.
    \]
\end{definition}
    We can regard an element $S\in\Fin(\alpha\times\beta)$ as the position $\{\prec S\}$ of $(\alpha\times\beta)$-chomp.
A function $\size:\Fin(\omega\times\alpha)\to [0,\omega\otimes\alpha]$ is defined in \cite[Section 2]{HS02}.
It satisfies the following conditions:
\begin{enumerate}
    \item $\size(\emptyset)=\omega\otimes\alpha$.
    \item If $S<T$, then $\size S<\size T$.
\end{enumerate}
\begin{definition}
    Let $M$ be a module over a ring $A$.
    Let $M[X]=\{x_0+\cdots+x_{d-1}X^{d-1}:x_0,\ldots,x_{d-1}\in M\}$ denote the polynomial module regarded as an $A[X]$-module.
    We regard a list $[x_0,\ldots,x_{d-1}]\in\List M$ as a polynomial $x_0+\cdots+x_{d-1}X^{d-1}\in M[X]$. Let $\deg{[x_0,\ldots,x_{d-1}]}:=d-1$.
    If $d\ge1$, let $\lc{[x_0,\ldots,x_{d-1}]}:=x_{d-1}$.
\end{definition}
\begin{definition}\label{describe}
Let $\alpha$ be an ordinal and $M$ be a module over a ring $A$.
    Let $\sigma:=[f_0,\ldots,f_{n-1}]\in\List(\List M)$. We say that \emph{$S\in\Fin(\omega\times\alpha)$ describes $\sigma$} if for every $(d,\alpha')\in S$, there exist $l\in\N$ and $g_0,\ldots,g_{l-1}\in\List M$ such that
        \begin{enumerate}
            \item $g_0,\ldots,g_{l-1}\in\gen{\sigma}_{M[X]}$,
            \item $\deg g_0,\ldots,\deg g_{l-1}\le d$, and
            \item $[\lc g_0,\ldots,\lc g_{l-1}]_M$ is $\alpha'$-good.
        \end{enumerate}
\end{definition}
\begin{lemma}\label{list-game}
    Let $\alpha$ be an ordinal and $M$ be an $\alpha$-Noetherian module over a ring $A$. Let $n\in\N$, $f_0,\ldots,f_{n-1}\in\List M$, and $S:=\{(d_0,\alpha_0),\ldots,(d_{m-1},\alpha_{m-1})\}\in\Fin(\omega\times\alpha)$ be an element that describes $\sigma:=[f_0,\ldots,f_{n-1}]$.
    Then, for every $f\in\List M$,
    \begin{enumerate}
        \item $\sigma.f\in\List M[X]$ is good, or
        \item there exists $S'\in\Fin(\omega\times\alpha)$ such that $S'$ describes $\sigma.f$ and $S'<S$.
    \end{enumerate}
\end{lemma}
\begin{proof}
    We prove this by induction on $\deg f$.
    \begin{enumerate}
        \item If $\deg f=-1$, then $f=_{M[X]}0$ and $\sigma.f$ is good.
        \item If $\deg f\ge0$, let $\alpha':=\min\{\gamma\in\{\alpha,\alpha_0,\ldots,\alpha_{m-1}\}:(\deg f,\gamma)\notin\{\prec S\}\}$. Then, there exist $l\in\N$ and $g_0,\ldots,g_{l-1}\in\List M$ such that
        \begin{enumerate}
            \item $g_0,\ldots,g_{l-1}\in\gen{\sigma}_{M[X]}$,
            \item $\deg g_0,\ldots,\deg g_{l-1}\le\deg f$, and
            \item $[\lc g_0,\ldots,\lc g_{l-1}]_M$ is $\alpha'$-good.
        \end{enumerate}
        Hence there exists $\beta\in[-1,\alpha')$ such that $[\lc g_0,\ldots,\lc g_{l-1},\lc f]\in\List M$ is $\beta$-good.
    \begin{enumerate}
        \item If $\beta=-1$, then $[\lc g_0,\ldots,\lc g_{l-1},\lc f]_M$ is good.
        Hence there exists $g\in\List M$ such that $\deg g=\deg f-1$ and $g-f\in\gen{g_0,\ldots,g_{l-1}}\subseteq\gen{\sigma}$.
        By the inductive hypothesis,
        \begin{enumerate}
            \item $\sigma.g\in\List M[X]$ is good, or
            \item 
        there exists $S'\in\Fin(\omega\times\alpha)$ such that $S'$ describes $\sigma.g$ and $S'<S$.
        \end{enumerate}
        If (i) holds, then $\sigma.f$ is good.
        If (ii) holds, then $S'$ describes $\sigma.f$, and $S'<S$.
        \item If $\beta\in[0,\alpha')$, then let $S':=S\cup\{(\deg f,\beta)\}$. Then $S'$ describes $\sigma.f$, and $S'<S$.\qedhere
    \end{enumerate}
    \end{enumerate}
\end{proof}
\begin{lemma}\label{size-good}
    Let $\alpha$ be an ordinal and $M$ be an $\alpha$-Noetherian module over a ring $A$. Let $n\in\N$, $f_0,\ldots,f_{n-1}\in\List M$, and $S:=\{(d_0,\alpha_0),\ldots,(d_{m-1},\alpha_{m-1})\}\in\Fin(\omega\times\alpha)$ be an element which describes $\sigma:=[f_0,\ldots,f_{n-1}]$.
    Then, $\sigma\in\List M[X]$ is $(\size S)$-good.
\end{lemma}
\begin{proof}
    Let $f\in\List M$.
    By induction on $\size S$, we prove that there exists $\beta\in[-1,\size S)$ such that $\sigma.f\in\List M[X]$ is $\beta$-good.
    By \cref{list-game},
    \begin{enumerate}
        \item $\sigma.f\in\List M[X]$ is good, or
        \item there exists $S'\in\Fin(\omega\times\alpha)$ such that $S'$ describes $\sigma.f$ and $S'<S$.
    \end{enumerate}
    If (1) holds, then $\sigma.f$ is $(-1)$-good.
    If (2) holds, then $\size S'<\size S$, and $\sigma.f$ is $(\size S')$-good by the inductive hypothesis.
    Hence there exists $\beta\in[-1,\size S)$ such that $\sigma.f$ is $\beta$-good.

    Hence $\sigma$ is $(\size S)$-good.
\end{proof}
\begin{theorem}\label{qhbt}
    Let $\alpha$ be an ordinal and $M$ be an $\alpha$-Noetherian module over a ring $A$. Then $M[X]$ is an $(\omega\otimes\alpha)$-Noetherian $A[X]$-module.
\end{theorem}
\begin{proof}
    Since $\emptyset\in\Fin(\omega\times\alpha)$ describes $[]\in\List(\List M)$, the list $[]\in\List(M[X])$ is $\size(\emptyset)$-good.
    Since $\size(\emptyset)=\omega\otimes\alpha$, the module $M[X]$ is $(\omega\otimes\alpha)$-Noetherian.
\end{proof}
The following corollary follows from \cref{qhbt} and the definition of $\alpha$-Noetherian rings:
\begin{corollary}[{Classically proved in \cite[Theorem 3.1]{Bro03}}]
    If $A$ is an $\alpha$-Noetherian ring, then the ring $A[X]$ is $(\omega\otimes\alpha)$-Noetherian.
\end{corollary}
By \cref{ex-noe}, we obtain a new proof of the following well-known constructive results on the Krull dimension:
\begin{example}
    \begin{enumerate}
        \item If $K$ is a discrete field, then $K[X_1,\ldots,X_n]$ is $\omega^n$-Noetherian. In particular, we have $\Kdim K[X_1,\ldots,X_n]<1+n$. This bound on the Krull dimension is constructively proved in \cite[Corollary 4]{CL05} and \cite[Theorem XIII-5.1]{LQ15}.
        \item The ring $\Z[X_1,\ldots,X_n]$ is $\omega^{1+n}$-Noetherian. In particular, we have $\Kdim \Z[X_1,\ldots,X_n]<2+n$. This bound on the Krull dimension is constructively proved in \cite[Theorem XIII-8.20]{LQ15}
    \end{enumerate}
\end{example}
\section{Length of Noetherian modules}
In this section, we reason in ZFC.
We prove that the notion of $\alpha$-Noetherian ring can be written in terms of the length of Noetherian modules defined by Gulliksen \cite{Gul73}.
First, we recall the definition of the length.
\begin{definition}
    Let $M$ be a Noetherian module over a ring $A$. For a submodule $N\le M$, we inductively define an ordinal $\lambda_M(N)$ by
    \[
    \lambda_M(N):=\sup\{\lambda_M(L)+1:N\lneq L\leq M\}.
    \]
    The ordinal $l(M):=\lambda_M(0)$ is called \emph{the length of $M$}.
\end{definition}
\begin{lemma}\label{alpha-N-is-N-lemma}
    Let $\alpha$ be an ordinal, $M$ be a module over a ring $A$, and $M_n$ $(n\in\N)$ be an ascending chain of submodules of $M$.
    If there exists an $\alpha$-good list $\sigma\in\List M$ such that $\gen{\sigma}\subseteq M_0$, then there exists $n\in\N$ such that $M_n=M_{n+1}$.
\end{lemma}
\begin{proof}
    We prove this by induction on $\alpha$.
    \begin{enumerate}
        \item If $\alpha=0$, then $M_0=M_1$.
        \item Let $\alpha>0$. We have $M_0=M_1$ or $M_0\subsetneq M_1$.
        \begin{itemize}
            \item If $M_0\subsetneq M_1$, then there exists $x\in M_1$ such that $x\notin M_0$. Since $\sigma$ is $\alpha$-good, there exists $\beta\in[-1,\alpha)$ such that $\sigma.x$ is $\beta$-good.
            Since $x\notin M_0$, we have $x\notin\gen{\sigma}$. Hence $\beta\ne-1$.
            We have $\gen{\sigma.x}\subseteq M_1$.
            Hence, by the inductive hypothesis, there exists $n\ge1$ such that $M_n=M_{n+1}$.
        \end{itemize}
        Hence there exists $n\in\N$ such that $M_n=M_{n+1}$.\qedhere
    \end{enumerate}
\end{proof}
\begin{proposition}\label{alpha-N-is-N}
    Let $\alpha$ be an ordinal, and $M$ be an $\alpha$-Noetherian module over a ring $A$. Then $M$ is Noetherian.
\end{proposition}
\begin{proof}
    Let $\sigma:=[]$ in \cref{alpha-N-is-N-lemma}.
\end{proof}
\begin{lemma}\label{alpha-N-is-alpha-lemma}
    Let $\alpha$ be an ordinal, $M$ be an $\alpha$-Noetherian module over a ring $A$, and $N$ be a submodule of $M$.
    If there exists an $\alpha$-good list $\sigma\in\List M$ such that $\gen{\sigma}\subseteq N$, then $\lambda_M(N)\le\alpha$.
\end{lemma}
\begin{proof}
    We prove this by induction on $\alpha$.
    \begin{itemize}
        \item Let $L\leq M$ be a submodule such that $N\lneq L$. Then, there exists $x\in L$ such that $x\notin N$. Since $\sigma$ is $\alpha$-good, there exists $\beta\in[-1,\alpha)$ such that $\sigma.x$ is $\beta$-good. Since $x\notin N$, We have $\beta\ne-1$. Hence, by the inductive hypothesis, $\lambda_M(L)\le\beta$.
    Hence $\lambda_M(L)+1\le\alpha$.
    \end{itemize}
    Hence $\lambda_M(N)\le\alpha$.
\end{proof}
\begin{proposition}\label{alpha-N-is-alpha}
    Let $\alpha$ be an ordinal, and $M$ be an $\alpha$-Noetherian module over a ring $A$. Then $l(M)\le\alpha$.
\end{proposition}
\begin{proof}
    Let $N:=0$ in \cref{alpha-N-is-alpha-lemma}.
\end{proof}
\begin{lemma}\label{N-is-alpha-N-lemma}
    Let $M$ be a Noetherian module over a ring $A$, and $\sigma\in\List M$.
    Then, $\sigma$ is $\lambda_M(\gen{\sigma})$-good.
\end{lemma}
\begin{proof}
    We prove this by induction on $\lambda_M(\gen{\sigma})$.
    \begin{itemize}
        \item 
    Let $x\in M$.
    \begin{enumerate}
        \item If $x\in\gen{\sigma}$, then $\sigma.x$ is good.
        \item If $x\notin\gen{\sigma}$, then $\lambda_M(\gen{\sigma.x})+1\le\lambda_M(\gen{\sigma})$.
        Hence, by the inductive hypothesis, $\sigma.x$ is $\lambda_M(\gen{\sigma.x})$-good.
    \end{enumerate}
    \end{itemize}
    Hence $\sigma$ is $\lambda_M(\gen{\sigma})$-good.
\end{proof}
\begin{proposition}\label{N-is-alpha-N}
    Let $M$ be a Noetherian module over a ring $A$.
    Then, $M$ is $l(M)$-Noetherian.
\end{proposition}
\begin{proof}
    Let $\sigma:=[]$ in \cref{N-is-alpha-N-lemma}.
\end{proof}
\cref{alpha-N-is-N}, \cref{alpha-N-is-alpha}, and \cref{N-is-alpha-N} together imply the following theorem.
\begin{theorem}\label{length}
    Let $\alpha$ be an ordinal, and $M$ be a module over a ring $A$. Then $M$ is $\alpha$-Noetherian if and only if $M$ is Noetherian and $l(M)\le\alpha$.
\end{theorem}
\section*{Acknowledgments}
The author would like to express his deepest gratitude to his supervisor, Ryu Hasegawa, for his support.
The author would like to thank Thierry Coquand for his helpful advice.
The author would also like to thank Yuto Arai, Ryuya Hora, Sangwoo Kim, Koki Kurahashi, Yuto Ikeda, Hiromasa Kondo, Haruki Toyota, and Minoru Sekiyama for their helpful comments in the graduate seminar.

This research was supported by Forefront Physics and Mathematics Program to Drive Transformation (FoPM), a World-leading Innovative Graduate Study (WINGS) Program, the University of Tokyo.


\begin{thebibliography}{99}
\bibitem[Bro02]{Bro02} G. Brookfield. The length of Noetherian modules. \emph{Comm. Algebra} 30(7):3177--3204, 2002.
\bibitem[Bro03]{Bro03} G.~Brookfield. The length of Noetherian polynomial rings. \emph{Comm. Algebra} 31(11):5591--5607, 2003.
\bibitem[Bru73]{Bru73} H. H. Brungs, Left Euclidean rings, Pacific Journal of Mathematics 45:27--33, 1973.
\bibitem[BSB23]{BSB23} G.~Buriola, P.~Schuster, I.~Blechschmidt. A Constructive Picture of Noetherian Conditions and Well Quasi-orders. In: G.~Della Vedova, B.~Dundua, S.~Lempp, F.~Manea, eds., \emph{Unity of Logic and Computation}, 50--62. CiE 2023. Lecture Notes in Comput. Sci., Vol.~13967. Springer, Cham, 2023.
\bibitem[CC06]{CC06} P. Cast\'eran and E. Contejean. On ordinal notations. User Contributions to the Coq Proof Assistant, 2006. \url{https://github.com/coq-contribs/cantor}.
\bibitem[Cla15]{Cla15} P. L. Clark. A Note on Euclidean Order Types. \emph{Order} 32(2):157--178, 2015.
\bibitem[Coh61]{Coh61} P. M. Cohn, On a generalization of the Euclidean algorithm. \emph{Proc. Cambridge Phil. Soc.} 57:18--30, 1961.
\bibitem[CNT19]{CNT19} C. J. Conidis, P. P. Nielsen, V. Tombs, Transfinitely valued Euclidean domains have arbitrary indecomposable order type. \emph{Comm. Algebra} 47(3):1105–1113, 2019.
\bibitem[CL05]{CL05} T. Coquand, H. Lombardi. A short proof for the Krull dimension of a polynomial ring. \emph{Am. Math. Mon}. 112(9):826-829, 2005.
\bibitem[CP99]{CP99} T. Coquand, H. Persson. Gröbner bases in type theory. In: T.~Altenkirch, B.~Reus, W.~Naraschewski, eds., \emph{TYPES 1998: Types for Proofs and Programs}. Lecture Notes in Comput. Sci., Vol. 1657, 33--46. Springer, 1999.
\bibitem[Fle71]{Fle71} C. R. Fletcher, Euclidean rings. \emph{J. London Math. Soc.} 4:79--82, 1971.
\bibitem[Gri13]{Gri13} J. Grimm. Implementation of three types of ordinals in Coq. Technical Report RR-8407, INRIA, 2013.
\bibitem[Gul73]{Gul73} T. H. Gulliksen. A theory of length for Noetherian modules. \emph{J. Pure Appl. Algebra} 3(2):159--170, 1973.
\bibitem[Hib75]{Hib75} J. J. Hiblot, Des anneaux euclidiens dont le plus petit algorithme n’est pas \`a valeurs finies.
\emph{C. R. Acad. Sci. Paris S\'er. A-B} 281(12), Ai, A411--A414, 1975.
\bibitem[Hib77]{Hib77} J. J. Hiblot, Correction \`a une note sur les anneaux euclidiens. \emph{C. R. Acad. Sci. Paris S\'er. A-B} 284(15), A847--A849, 1977.
\bibitem[HS02]{HS02} S. Huddleston, J. Shurman. Transfinite Chomp. In: R.~J.~Nowakowski, ed., \emph{More games of no chance}. Math. Sci. Res. Inst. Publ., Vol. 42, 183--212. Cambridge University Press, Cambridge, 2002.
\bibitem[JL91]{JL91} C. Jacobsson, C. Löfwall. Standard bases for general coefficient rings and a new constructive proof of Hilbert's basis theorem. \emph{J. Symbolic Comput.} 12(3):337--371, 1991.
\bibitem[KNFX21]{KNFX21} N.~Kraus, F.~Nordvall Forsberg, C.~Xu. Connecting constructive notions of ordinals in homotopy type theory. In \emph{46th International Symposium on Mathematical Foundations of Computer Science (MFCS 2021)}, Leibniz International Proceedings in Informatics (LIPIcs), 202, 70:1--70:16. Schloss Dagstuhl -- Leibniz-Zentrum f\"ur Informatik, 2021.
\bibitem[KNFX23]{KNFX23} N.~Kraus, F.~Nordvall Forsberg, C.~Xu. Type-theoretic approaches to ordinals. \emph{Theoret. Comput. Sci.} Vol. 957, Paper No.~113843, 2023.
\bibitem[Len74]{Len74} H. W. Lenstra. \emph{Lectures on Euclidean Rings}, Bielefeld, Summer 1974.
\bibitem[LC14]{LC14} J. Liu, M. Chen. Euclidean modules, \emph{Math. Notes}, Volume 95(5--6):865--872, 2014.
\bibitem[Lom02]{Lom02} H. Lombardi. Dimension de Krull, Nullstellensätze et évaluation dynamique. \emph{Math. Z.} 242(1):23--46, 2002.
\bibitem[Lom23]{Lom23} H. Lombardi. \emph{Dimension de Krull, Nullstellensätze et évaluation dynamique}. (arXiv version). \texttt{arXiv:2308.10296v1 [math.AC]}, 2023.
\bibitem[LQ15]{LQ15} H. Lombardi, C. Quitté. \emph{Commutative Algebra: Constructive Methods}. (T.~K. Roblot, trans.) Algebra and Applications, Vol. 20. Dordrecht, The Netherlands: Springer, 2015.
\bibitem[MRR88]{MRR88} R.~Mines, F.~Richman, W.~Ruitenburg. \emph{A course in constructive algebra}. Universitext. Springer-Verlag, New York, 1988.
\bibitem[Mot49]{Mot49} T. Motzkin, The Euclidean algorithm. \emph{Bull. Amer. Math. Soc.} 55:1142--1146, 1949.
\bibitem[Nag78]{Nag78} M. Nagata, On Euclid algorithm. \emph{C. P. Ramanujam---a Tribute}, Tata Inst. Fund. Res. Studies in Math. 8:175--186, 1978.
\bibitem[Nag85]{Nag85} M. Nagata, Some remarks on Euclid rings. \emph{J. Math. Kyoto Univ.} 25:421--422, 1985.
\bibitem[NFXG20]{NFXG20} F. Nordvall Forsberg, C. Xu, and N. Ghani. Three equivalent ordinal notation systems in cubical Agda. \emph{CPP’20}, 172--185. ACM, 2020.
\bibitem[Ore33]{Ore33} O. Ore, Theory of non commutative polynomials, Ann. of Math. 34:480--508, 1933.
\bibitem[Rah02]{Rah02} A. M. Rahimi, Euclidean modules. \emph{Libertas Math.} 22:123--126, 2002.
\bibitem[Ric74]{Ric74} F. Richman. Constructive aspects of Noetherian rings. \emph{Proc. Amer. Mat. Soc.} 44(2):436--441, 1974.
\bibitem[Sam71]{Sam71} P. Samuel, About Euclidean rings. \emph{J. Algebra} 19:282--301, 1971.
\bibitem[Sei74]{Sei74} A. Seidenberg. What is Noetherian? \emph{Seminario Mat. e. Fis. di Milano} 44(1):55--61, 1974.
\end{thebibliography}
\end{document}